\documentclass[12pt]{article}

\usepackage{amsthm,amssymb,amsmath,color}
\usepackage[usenames,dvipsnames]{pstricks}
\usepackage{epsfig}
\usepackage{enumerate}
\usepackage{comment}
\usepackage{tikz}
\usetikzlibrary{intersections,arrows}
\usepackage{hyperref}

\newtheorem{thm}{Theorem}[section]
\newtheorem{prop}[thm]{Proposition}
\newtheorem{lem}[thm]{Lemma}

\DeclareMathOperator{\Aff}{Aff}

\title{Distributive and anti-distributive Mendelsohn triple systems}
\author{Diane M. Donovan\thanks{Centre for Discrete Mathematics and Computing, University of Queensland, St Lucia 4072, AUSTRALIA}\,\,, 
Terry S. Griggs\thanks{Department of Mathematics and Statistics, The Open University, Walton Hall, Milton Keynes MK7 6AA, UNITED KINGDOM}\,\,,
  Thomas A. McCourt\thanks{School of Computing and Mathematics, Plymouth University, Drake Circus, Plymouth PL4 8AA, UNITED KINGDOM
 and Heilbronn Institute for Mathematical Research, University of Bristol, University Walk, Bristol BS8 1TW, UNITED KINGDOM  }\,\,,\\
 Jakub Opr\v{s}al\thanks{Department of Algebra, Faculty of Mathematics and Physics, Charles University, Sokolovsk\'{a} 83, 18675 Praha 8, CZECH REPUBLIC}\,\, and
David Stanovsk\'{y}\thanks{Department of Algebra, Faculty of Mathematics and Physics, Charles University, Sokolovsk\'{a} 83, 18675 Praha 8, CZECH REPUBLIC}}
\date{\small Keywords: Mendelsohn triple system, quasigroup, distributive, Moufang loop, Loeschian numbers.
Research partially supported by the GA\v CR grant 13-01832S (Opr\v sal, Stanovsk\'y).}

\begin{document}
\maketitle

\begin{abstract}
We prove that the existence spectrum of Mendelsohn triple systems whose associated quasigroups satisfy distributivity corresponds to the Loeschian numbers, and provide some enumeration results. We do this by considering a description of the quasigroups in terms of commutative Moufang loops.

In addition we provide constructions of Mendelsohn quasigroups that fail distributivity for as many combinations of elements as possible. 

These systems are analogues of Hall triple systems and anti-mitre Steiner triple systems respectively.
\end{abstract}

\section{Introduction}
\label{sec:intro}

\subsection{Steiner and Mendelsohn triple systems}

Hall triple systems and anti-mitre systems are important and well-known types of Steiner triple systems. The aim of this paper is to introduce and prove the existence of analogous systems in the context of Mendelsohn triple systems. The natural concept for doing this is that of distributivity in the associated quasigroups. 
In the distributive case we make use of results from the general theory of commutative Moufang loops, in particular the Fischer-Smith-Galkin classification of finite distributive quasigroups.
First we define the terms that will be used.

A \textit{Steiner triple system} of order $v$, usually denoted by STS($v$), is an ordered pair $(V,\mathcal{B})$ where $V$ is a \textit{base set} of \textit{elements} or \textit{points} of cardinality $v$ and $\mathcal{B}$ is a collection of 3-element subsets of $V$, called \textit{blocks}, which collectively have the property that every pair of distinct elements of $V$ is contained in precisely one block. An STS($v$) exists if and only if $v\equiv 1\text{ or } 3\,(\text{mod }6)$, \cite{Kirk}.

A \textit{totally symmetric quasigroup}, or a \textit{Steiner quasigroup}, is an idempotent quasigroup $(V,\circ)$ satisfying equations $x\circ y=y\circ x$ and $x\circ(y\circ x)=y$ for every $x,y\in V$. In terms of translations, the two conditions are equivalent to $L_x=R_x$ and $L_xR_x=I$, for every $x\in V$. Here $L_x(y)=x\circ y$ and $R_x(y)=y\circ x$ are the \emph{left} and \emph{right translations}, respectively, and $I$ denotes the identity mapping.
Given an STS($v$), a Steiner quasigroup can be formed by defining an operation $\circ$ on the set $V$ using the rules $x\circ x=x$, for all $x\in V$, and $x\circ y=z$,  for all $x,y\in V$ where $\{x,y,z\}\in\mathcal{B}$. We say that the Steiner quasigroup so formed is associated with the Steiner triple system. Also note that starting with a Steiner quasigroup one can reverse the process to construct a Steiner triple system. 

A \textit{Mendelsohn triple system} of order $v$, usually denoted by MTS($v$), is an ordered pair $(V,\mathcal{B})$ where $V$ is a \textit{base set} of \textit{elements} or \textit{points} of cardinality $v$ and $\mathcal{B}$ is a collection of \textit{cyclically ordered blocks} $\langle x,y,z\rangle$ which collectively have the property that every \textit{ordered pair} of distinct elements is contained in a unique block, i.e., the above block contains the ordered pairs $(x,y),(y,z)$ and $(z,x)$. An MTS($v$) exists if and only if $v\equiv 0\text{ or } 1\,(\text{mod }3)$, $v\neq 6$, \cite{Men}.
Let $(V,\mathcal{B})$ be an MTS($v$). If $\langle a,b,c \rangle\in\mathcal{B}$ implies that $\langle a,c,b \rangle\in\mathcal{B}$, then the Mendelsohn triple system is formed from a Steiner triple system by writing each block of the STS($v$) in both of its two cyclic orders. An MTS($v$) which is not formed in this way will be called \textit{proper}. 

A \textit{semi-symmetric quasigroup}, or a \textit{Mendelsohn quasigroup}, is an idempotent quasigroup $(V,\circ)$ satisfying the equation $x\circ(y\circ x)=y$ for every $x,y\in V$. In terms of translations, the condition is equivalent to $L_xR_x=I$, for every $x\in V$.
Given an MTS($v$), a Mendelsohn quasigroup on $V$ can be formed in an analogous manner to the Steiner case, by the rules $x\circ x=x$, for all $x\in V$ and $x\circ y=z$, for all $x,y\in V$ where $\langle x,y,z\rangle\in\mathcal{B}$. 
If an MTS($v$) is not proper, then its associated Mendelsohn quasigroup is commutative and is identical to the Steiner quasigroup associated with the STS($v$) which forms the MTS($v$) in the above manner.

\subsection{Distributivity}

Important subclasses of Steiner triple systems are the \textit{affine Steiner triple systems} and the \textit{Hall triple systems}. As we shall see, these are the Steiner triple systems where the associated quasigroups are, respectively, medial and distributive. They can be defined as follows.
\begin{itemize}
\item[(i)] Let $\mathbb{F}_3$ be the field of three elements and $V=(\mathbb{F}_3)^n$. Let $\mathcal{B}$ be the set of blocks $\{\mathbf{x},\mathbf{y},\mathbf{z}\}$ where $\mathbf{x},\mathbf{y},\mathbf{z}\in V$, $\mathbf{x}+\mathbf{y}+\mathbf{z}=\mathbf{0}$ and $\mathbf{x}\neq\mathbf{y}\neq\mathbf{z}\neq\mathbf{x}$. This is the affine Steiner triple system AG($n,3$) of order $3^n$.
The associated Steiner quasigroup is $((\mathbb F_3)^n,\circ)$ where $\mathbf x\circ\mathbf y=-\mathbf x-\mathbf y$.
\item[(ii)] Hall triple systems were introduced in \cite{Hall} as Steiner triple systems in which for each $x\in V$, the automorphism group contains an involution with just $x$ as a fixed point. They can be characterised as Steiner triple systems in which every three points which do not form a block generate the affine triple system AG($2,3$) of order 9. Hall triple systems have order $3^m$, $m\geq 2$, and the class of Hall triple systems contains the class of affine Steiner triple systems. The smallest non-affine Hall triple system has order 81.
\end{itemize}

We start our account on distributivity with a well-known quasigroup construction.
Let $(G,+)$ be an Abelian group, and suppose that $k$ is an automorphism of $(G,+)$ such that $I-k$ is also an automorphism. Then $Q=(G,*_k)$ where 
$$x*_k y=(I-k)(x)+k(y),$$ for all $x,y\in G$, is an idempotent quasigroup. Such a quasigroup $Q$ is called an \textit{affine quasigroup} and is denoted $\Aff(G,k)$. 
For example, the quasigroup associated to AG($n,3$) is $\Aff((\mathbb Z_3)^n,-I)$. In Proposition \ref{prop:affine} we show that an affine quasigroup $\Aff(G,k)$ is
\begin{enumerate}
	\item[(S)] a Steiner quasigroup, if and only if the exponent of $G$ is 3 and $k=-I$; and
	\item[(M)] a Mendelsohn quasigroup, if and only if $k$ satisfies $k-k^2=I$.
\end{enumerate}
We will say a Mendelsohn triple system is \textit{affine} if its associated quasigroup is affine.

Affine quasigroups admit a convenient equational characterisation.
A quasigroup $(Q,\circ)$ is \textit{medial} if $(x\circ y)\circ(u\circ v)=(x\circ u)\circ(y\circ v)$, for all $x,y,u,v\in Q$. 
A special case of the well-known Toyoda-Bruck Theorem \cite{Bru} is the following characterisation.

\begin{thm}[Toyoda \& Bruck] \label{thm:toyoda-bruck}
Let $Q$ be an idempotent quasigroup. Then $Q$ is medial if and only if $Q$ is affine, i.e., isomorphic to some $\Aff(G,k)$.
\end{thm}

A quasigroup $(Q,\circ)$ is \textit{left distributive} if $x\circ (y\circ z)=(x\circ y)\circ(x\circ z)$, for all $x,y,z\in Q$, i.e., if the left translation $L_x$ is an automorphism, for every $x\in Q$. Dually, it is \textit{right distributive} if $(x\circ y)\circ z=(x\circ z)\circ(y\circ z)$, for all $x,y,z\in Q$, i.e., if the right translation $R_z$ is an automorphism, for every $z\in Q$. Notice that medial idempotent quasigroups are (both left and right) distributive and that distributive quasigroups are idempotent. In Steiner and Mendelsohn quasigroups, the left and right distributivity are equivalent properties: we have $L_x=R_x^{-1}$, hence $L_x$ is an automorphism if and only if $R_x$ is an automorphism.

Belousov \cite{Bel}, Theorem 8.6, provides an important characterisation of distributivity: an idempotent quasigroup is distributive if and only if every 3-generated subquasigroup is medial. In the context of Steiner (or Mendelsohn) triple systems, 3-generated subquasigroups correspond to 3-generated subsystems.
Hence, a quasigroup associated to a Steiner (or Mendelsohn) triple system $(V,\mathcal{B})$ is distributive if and only if every 3-generated subsystem of $(V,\mathcal{B})$ is affine. Observing that a Steiner quasigroup $\Aff(G,-I)$ has at most 3 generators if and only if $G=(\mathbb Z_3)^n$ with $n\leq2$, we obtain the following well known theorem, \cite{CRC}, Theorem 28.15, page 497:
a quasigroup associated to a Steiner triple system $(V,\mathcal{B})$ is distributive if and only if $(V,\mathcal{B})$ is a Hall triple system.
In the Mendelsohn setting, the situation is more complicated, since there are many 3-generated Mendelsohn quasigroups.

The affine representation generalises to distributive quasigroups, by allowing $G$ to be a more general structure, a commutative Moufang loop.
A \textit{commutative Moufang loop} $(G,+)$ is a commutative quasigroup that contains an identity element and satisfies the equation $(x+x)+(y+z)=(x+y)+(x+z)$, for all $x,y,z\in G$. The \textit{nucleus} $N(G,+)$ is the subset of $G$ whose elements associate with all elements of $G$. An automorphism $k$ of $(G,+)$ is \textit{nuclear} if $(I+k)(x)=x+k(x)\in N(G,+)$ for all $x\in G$. Starting with a commutative Moufang loop $(G,+)$ and a nuclear automorphism $k$ such that $I-k$ is also an automorphism, an idempotent quasigroup $Q=(G,*_k)$ is described by $x*_k y=(I-k)(x)+k(y),$ for all $x,y\in G$, and is said to be \emph{affine over a commutative Moufang loop}. We will use the notation $\Aff(G,k)$ as in the case of Abelian groups. Conditions (S) and (M) hold similarly, again see Proposition \ref{prop:affine}.

Distributive quasigroups are explicitly described by the Belousov-Soublin Theorem that appeared implicitly in \cite{Bel}, Section VIII.2, and explicitly in \cite{Sob}, Section II.7, Theorem 1.

\begin{thm}[Belousov \& Soublin] \label{thm:belousov-soublin}
Let $Q$ be a quasigroup. Then $Q$ is distributive if and only if $Q$ is affine over a commutative Moufang loop.
\end{thm}

A deep theory of commutative Moufang loops has been developed over the years \cite{Ben}. In particular, directly indecomposable non-associative commutative Moufang loops of order $n$ exist if and only if $n=3^k$ with $k\geq4$ (cf. the existence spectrum of non-affine Hall triple systems). We refer to \cite{Sta}, Section 3, for details on the affine representation theory for distributive quasigroups, including proofs of Theorems \ref{thm:toyoda-bruck} and \ref{thm:belousov-soublin}.
We will use one of the consequences of the general theory, the classification of finite distributive quasigroups (we use Galkin's interpretation of the so called Fischer-Smith theorem, \cite{Gal,Smith}).

\begin{thm} [Fischer \& Smith \& Galkin]
\label{thm:Galkin_Smith}
Let $v=p_1^{r_1}\ldots p_a^{r_a}$ where $p_1,\ldots, p_a$ are pairwise distinct primes and let $Q$ be a distributive quasigroup of order $v$. Then $Q$ is isomorphic to a direct product of distributive quasigroups $Q_1\times \ldots\times Q_a$ where $|Q_i|=p_i^{r_i}$. Moreover, if $Q_i$ is not affine, then $p_i=3$ and $r_i\geq 4$. 
\end{thm}

Section 2 is structured as follows. 
In Subsection \ref{subsec:affine}, we discuss the conditions characterising affine Steiner and Mendelsohn quasigroups (Proposition \ref{prop:affine}).
Subsection \ref{subsec:spectrum} determines the existence spectrum of distributive Mendelsohn quasigroups (Theorems \ref{thm:sufficient} and \ref{thm:necessary}).
In Subsection \ref{subsec:enumeration} we provide some enumeraton results on distributive Mendelsohn quasigroups, including prime and prime squared orders (Theorem \ref{thm:enumeration}), and order $3^4=81$.

\subsection{Anti-distributivity}

It is also natural to ask a related question. Given a Steiner (respectively Mendelsohn) quasigroup, it is easily verified that all ordered triples $(x,y,z)$, where at least two of the elements are equal or where $\{x,y,z\}$ (respectively $\langle x,y,z\rangle$) is a block of $\mathcal{B}$, satisfy distributivity. But, do there exist Steiner (respectively Mendelsohn) triple systems  where all ordered triples $(x,y,z)$ of distinct points which are not blocks violate distributivity? We will refer to such systems as being \textit{anti-distributive}. Again for Steiner quasigroups the answer is known. 

In a Steiner triple system a collection or configuration of five blocks isomorphic to $\{z,b,x\}$, $\{z,g,c\}$, $\{z,a,y\}$, $\{b,g,a\}$, $\{x,c,y\}$ is called a \textit{mitre}. Diagrammatically it can be represented as shown in Figure \ref{fig:mitre}.

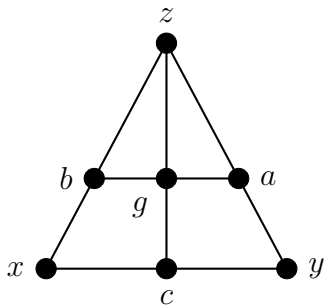
\begin{figure}[!hb]
\begin{center}
\begin{tikzpicture}[scale=0.8,>=stealth, vertex/.style={circle,inner sep=2.7,fill=black,draw}]

\coordinate (x) at (0,0);
\coordinate (c) at (2,0);
\coordinate (y) at (4,0);
\coordinate (b) at (0.8,1.5);
\coordinate (g) at (2,1.5);
\coordinate (a) at (3.2,1.5);
\coordinate (z) at (2,3.75);

\node at (x) [vertex,label=west:$x$]{};
\node at (c) [vertex,label=south:$c$]{};
\node at (y) [vertex,label=east:$y$]{};
\node at (b) [vertex,label=west:$b$]{};
\node at (g) [vertex,label=south west:$g$]{};
\node at (a) [vertex,label=east:$a$]{};
\node at (z) [vertex,label=north:$z$]{};

\draw [thick] (x) -- (z);
\draw [thick] (x) -- (y);
\draw [thick] (b) -- (a);
\draw [thick] (c) -- (z);
\draw [thick] (y) -- (z);
\end{tikzpicture}
\end{center}
\caption{Illustration of a mitre.}
\label{fig:mitre}
\end{figure}

There exist Steiner triple systems in which there are no mitres, so called \textit{anti-mitre} STS($v$). The distributive law describes the mitre; $c=x\circ y$, $b=x\circ z$, $a=y\circ z$, $g=c\circ z=(x\circ y)\circ z=b\circ a=(x\circ z)\circ(y\circ z)$. Thus, a Steiner quasigroup is anti-distributive if and only if the associated Steiner triple system is anti-mitre. 

Turning to Mendelsohn triple systems there appears to be no study of whether there exist Mendelsohn quasigroups that are anti-distributive. In Section \ref{sec:antidistributive} we give the first construction of such quasigroups that are associated with proper MTS($v$) for $v\equiv 3\text{ or }7\,(\text{mod }12)$, except for $v=19$. 

\section{Distributive Mendelsohn quasigroups}

\subsection{Affine Mendelsohn quasigroups}
\label{subsec:affine}

First, we show the conditions that characterise Steiner (respectively Mendelsohn) quasigroups that are affine over a commutative Moufang loop.
In the proof, we frequently use the well known property that commutative Moufang loops are diassociative \cite{Bru}, i.e., expressions involving only two elements do not depend on parenthesising.

\begin{prop}\label{prop:affine} 
Let $(G,+)$ be a commutative Moufang loop, and suppose that $k$ is a nuclear automorphism of $(G,+)$ such that $I-k$ is also an automorphism. Then 
$\Aff(G,k)$ is
\begin{enumerate}
	\item[(S)] a Steiner quasigroup, if and only if the exponent of $G$ is 3 and $k=-I$;
	\item[(M)] a Mendelsohn quasigroup, if and only if $k$ satisfies $k-k^2=I$.
\end{enumerate}
\end{prop}

\begin{proof}
Let $0$ be the unit element in $(G,+)$ and recall that, in $\Aff(G,k)$, we have $L_x(y)=R_y(x)=(I-k)(x)+k(y)$ for every $x,y\in G$.

(S) Suppose that $\Aff(G,k)$ is a Steiner quasigroup. As $L_0=R_0$ we have $L_0(x)=k(x)=(I-k)(x)=R_0(x)$, hence $2k(x)=x$, and thus $4k^2(x)=2k(2k(x))=x$.
As $L_0^2=I$ we have $L_0^2(x)=0*_k(0*_k x)=x$ and from the definition of the binary operation $*_k$, $k^2(x)=L_0^2(x)=x$. Thus $4x=x$, so the exponent of $G$ is 3. In particular $2k(x)=-k(x)$ and, as $2k(x)=x$, we have $k(x)=-x$, for every $x\in G$.

Now suppose that the exponent of $G$ is 3 and that $k=-I$. Then $L_x(y)=x+x-y=-x-y=-y-x=y+y-x=R_x(y)$ and $L_x^2(y)=-x-(-x-y)=y$, for every $x,y\in G$.

(M) Suppose that $\Aff(G,k)$ is a Mendelsohn quasigroup. Then, as $L_0R_0=I$, we have $x=L_0R_0(x)=L_0((I-k)(x))=(k-k^2)(x)$ for every $x\in G$.

Now suppose that $k$ satisfies $k-k^2=I$, hence also $k^2=k-I$. Then, for every $x,y\in G$,
\begin{align*}
L_xR_x(y) &= (I-k)(x)+k((I-k)(y)+k(x)) \\ &= (I-k)(x)+[k^2(x) + (k-k^2)(y)] \\ &= (I-k)(x)+[(k-I)(x) + I(y)] \\ &= y,
\end{align*}
using $(I-k)(x)=-(k-I)(x)$ in the last step.
\end{proof}

Note that the conditions on an automorphism $k$ from (S) or (M) also imply that $I-k$ is an automorphism. If $k=-I$ and the exponent of $G$ is 3, then $I-k=2I$ is an automorphism, and if $k-k^2=I$, then $I-k=-k^2$ is also an automorphism. Also note that $k=-I$ is always nuclear.

Condition (M) is related to the properties of the polynomial $f=x^2-x+1$. In particular, if $G=\mathbb Z_{p^d}$ is a cyclic group, then $\Aff(G,k)$ is a Mendelsohn quasigroup if and only if $k$ is a root of $f$ modulo $p^d$ (acting on $G$ as an automorphism, since then $p\nmid k$). The number of roots is determined in the next lemma, used later in our classification results.

\begin{lem} \label{lem:roots_of_f}
Let $p$ be a prime, $d\geq 1$ and $f = x^2 - x + 1$. Then $f$ has
\begin{enumerate}[(i)]
	\item two distinct roots modulo $p^d$ if $p\equiv 1 \pmod 3$;
	\item no roots modulo $p^d$ if $p\equiv 2 \pmod 3$;
	\item a double root modulo 3, and no roots modulo $3^d$ for $d>1$.
\end{enumerate}
\end{lem}

\begin{proof} 
First consider $d=1$.
The discriminant of $f$ is $-3$, so we get immediately that $f$ has a~double root modulo $p$ if and only if $p = 3$. Otherwise, the discriminant is not divisible by $p$, so we have either none, or two distinct roots.
If $p=2$, then $f$ has no roots. 
Suppose $p>3$ and let $a$ be a~root of $f$ in $\mathbb F_p$.
Since $a^3 = a^2 - a = -1$, the order of $a$ in $\mathbb F_p^*$ is $6$. If $p \equiv 2 \pmod 3$, then $6$ does not divide $|\mathbb F_p^*| = p-1$, contradiction. If $p\equiv1\pmod 3$, then $p \equiv 1 \pmod 6$, hence $6$ does divide $|\mathbb F_p^*|$. Let $k$ be the primitive sixth root of unity in $\mathbb F_p$. Then $k$ is a root of $x^6-1=(x^3-1)(x+1)(x^2-x+1)$, so $k$ is also the root of the polynomial $x^2-x+1$.

Now let $d>1$. Since there is no root modulo $p$ for any $p\equiv 2\pmod 3$, there is no root modulo $p^d$ either. Similarly, one readily checks that $f$ has no root modulo 9, hence no root modulo $3^d$. Finally, a standard Hensel lifting argument shows that there are two distinct roots of $f$ modulo $p^d$ for any $p\equiv 1\pmod 3$ and $d>1$. 
\end{proof}

\subsection{Existence spectrum}
\label{subsec:spectrum}

We start with a proof of the sufficient condition for existence of affine MTS($v$).

\begin{prop}
\label{thm:directprod}
The direct product of affine (over commutative Moufang loops) Mendelsohn quasigroups is an affine (over a commutative Moufang loop) Mendelsohn quasigroup.
\end{prop}

\begin{proof}
As they preserve all equations, the direct product of medial (respectively distributive) Mendelsohn quasigroups is a medial (respectively distributive) Mendelsohn quasigroup.  Thus, the statement follows from Theorems \ref{thm:toyoda-bruck} and \ref{thm:belousov-soublin}.
\end{proof}

\begin{lem}
\label{thm:menNet}
Let $p^d\equiv 1\,(\text{mod }6)$ where $p$ is a prime. Then there exists an affine MTS($p^d$).
\end{lem}

\begin{proof}
Let $\omega$ be a generator of the cyclic multiplicative group of the Galois field $\mathbb F_{p^d}$ of order $p^d-1=6s$. Let $k=\omega^s$ be a primitive sixth root of unity. Now, $x^6-1=(x^3-1)(x+1)(x^2-x+1)$, so $k$ is the root of the polynomial $x^2-x+1$. Hence as $k$ acts as an automorphism of the additive group $G$ of the field $\mathbb F_{p^d}$, it satisfies condition (M) of Proposition \ref{prop:affine}. Hence $\Aff(G,k)$ is the required example.
\end{proof}

\begin{lem}
\label{thm:menNet2}
There exists an affine MTS($2^{2d}$) for every $d\geq1$.
\end{lem}

\begin{proof}
Let $\omega$ be a generator of the cyclic multiplicative group of the Galois field $\mathbb F_{2^{2d}}$ of order $2^{2d}-1=3s$. Let $k=\omega^s$ be a primitive third root of unity. Since the field has characteristic 2, $k$ is a root of the polynomial $x^3+1=(x+1)(x^2+x+1)$, and thus also of the polynomial $x^2+x+1$. As in the previous proof, $\Aff(G,k)$ is the required example, where $G$ is the additive group of $\mathbb F_{2^{2d}}$.
\end{proof}

\begin{lem}
\label{thm:menNet3}
There exists an affine MTS($3^d$) for every $d\geq1$.
\end{lem}

\begin{proof}
An example is the Steiner quasigroup $\Aff((\mathbb Z_3)^d,-I)$.
\end{proof}

We are now in a position to state and prove the sufficient condition.

\begin{thm}
\label{thm:sufficient}
Let $v=p_1^{r_1}\ldots p_a^{r_a} q_1^{s_1}\ldots q_b^{s_b}$, where $p_1,\ldots, p_a,  q_1,\ldots, q_b$ are distinct primes, each $p_i\equiv 1\,(\text{mod }6)$ or $p_i=3$, and each $q_i\equiv 2\,(\text{mod }3)$. If each of the $s_i$, $1\leq i\leq b$ are even, then there exists an affine MTS($v$).
\end{thm}

\begin{proof}
Recursively applying Proposition \ref{thm:directprod} using Lemmas \ref{thm:menNet}, \ref{thm:menNet2} and \ref{thm:menNet3} obtains the result.
\end{proof}

Next we prove that the sufficient condition given in Theorem \ref{thm:sufficient} is also necessary.
We will make use of the following results.

\begin{lem}
\label{lem:2statement}
Let $p$ be a prime such that $p\equiv 2\,(\text{mod }3)$, let $d$ be odd, and $f$ an irreducible polynomial over $\mathbb F_p$ of even degree. Then there is no matrix $A\in GL(d,\mathbb{F}_p)$ such that $f(A)=0$.
\end{lem}

\begin{proof}
Assume there is a matrix $A\in GL(d,\mathbb{F}_p)$ such that $f(A)=0$. Since $f$ is irreducible, it is the minimal polynomial of $A$. Let $\chi$ be the characteristic polynomial of $A$. Then $f$ and $\chi$ have identical roots in the algebraic closure of $\mathbb{F}_p$, hence $\chi\mid f^n$ for some $n$. Since $f$ is irreducible, we have $\chi=f^m$ for some $m$. Now $d$, the size of the matrix $A$, is equal to the degree of its characteristic polynomial. But $\deg(\chi)=m\deg(f)$, contradicting the assumption that $d$ is odd.
\end{proof}
 
\begin{lem}\label{lem:necessary}
Let $p$ be prime such that $p\equiv 2\,(\text{mod }3)$ and let $d$ be odd. Then no distributive Mendelsohn quasigroup of order $p^d$ exists.
\end{lem}

\begin{proof}
First we show that there is no affine Mendelsohn quasigroup $\Aff((\mathbb Z_p)^d,k)$, for any $k$.
The automorphisms of the group $(\mathbb{Z}_p)^d$ are precisely the automorphisms of the vector space $(\mathbb{F}_p)^d$, i.e., elements of $GL(d,\mathbb{F}_p)$. 
By Lemma \ref{lem:roots_of_f}, the polynomial $f=1-x+x^2$ is irreducible over the field $\mathbb{F}_p$, hence from Lemma \ref{lem:2statement} there is no $A\in GL(d,\mathbb{F}_p)$ such that $f(A)=0$. Hence there is no Mendelsohn quasigroup $\Aff((\mathbb Z_p)^d,k)$ by Proposition \ref{prop:affine}.

We continue by induction. Let $\Aff(G,k)$ be an affine Mendelsohn quasigroup of order $p^d$ with the smallest possible odd $d$. Without loss of generality, $G=\prod_{i=1}^m \mathbb{Z}_{p^{d_i}}$ where $\sum_{i=1}^{m}d_i=d$, and we have that $d_i>1$ for at least one $i$. Let $H=\prod_{i=1}^m\langle p^{d_i-1}\rangle\leq G$, then $H\cong (\mathbb{Z}_p)^{m}$.
If $m$ is odd, then $\Aff(H,k|_H)$ is an affine Mendelsohn quasigroup of order $p^{m}$ where $m<d$, a contradiction. 
Assume that $m$ is even. Consider the group $G/H\cong\prod_{i=1}^{m} \mathbb{Z}_{p^{d_i-1}}$. Then $|G/H|=p^{\sum_{i=1}^{m}d_i-1}=p^{d-m}$, with $d-m$ odd. Moreover, $k/H$ and $(I-k)/H$ are automorphisms of $G/H$. Hence $\Aff(G/H,k/H)$ is an affine Mendelsohn quasigroup with a smaller odd exponent, again a contradiction.
\end{proof}

We are now in a position to state and prove the necessary condition.

\begin{thm}
\label{thm:necessary}
Let $v=p_1^{r_1}\ldots p_a^{r_a} q_1^{s_1}\ldots q_b^{s_b}$, where $p_1,\ldots, p_a,  q_1,\ldots, q_b$ are pairwise distinct primes, each $p_i\equiv 1\,(\text{mod }6)$ or $p_i=3$, and each $q_i\equiv 2\,(\text{mod }3)$. If some of the $s_i$, $1\leq i\leq b$ is odd, then no distributive Mendelsohn quasigroup of order $v$ exists.
\end{thm}

\begin{proof}
Assume $Q$ is a distributive Mendelsohn quasigroup of order $v$. According to Theorem \ref{thm:Galkin_Smith}, $Q$ is isomorphic to a direct product $Q_1\times\ldots Q_a\times R_1\times\ldots R_b$ of distributive Mendelsohn quasigroups (indeed, all fibres satisfy the equations satisfied by $Q$) such that $|Q_i|=p_i^{r_i}$ and $|R_i|=q_i^{s_i}$ for every $i$. Since all $q_i\neq 3$, by Theorem \ref{thm:Galkin_Smith} the quasigroups $R_i$ are affine, contradicting Lemma \ref{lem:necessary}.
\end{proof}

Theorems \ref{thm:sufficient} and \ref{thm:necessary} combine to provide necessary and sufficient conditions for the existence of distributive Mendelsohn triple systems. The existence spectrum is the set of Loeschian numbers, $\{x^2+xy+y^2:x,y\geq 1\}$, see the Encyclopaedia of Integer Sequences, \url{http://oeis.org/A003136}.

\subsection{Enumeration}
\label{subsec:enumeration}

Let $a(v)$ denote the number of isomorphism classes of affine Mendelsohn quasigroups of order $v$; let $b(v)$ denote the number of isomorphism classes of non-affine distributive Mendelsohn quasigroups of order $v$; and let $d(v)=a(v)+b(v)$. 

If $u$ and $v$ are coprime, the classification of finite Abelian groups implies that $a(uv)=a(u)a(v)$. Furthermore, the Galkin-Smith Theorem \ref{thm:Galkin_Smith} implies that $b(3^d v)=a(v)b(3^d)$ whenever $3\nmid v$ and $d\geq 0$, and that $b(1)=b(3)=b(9)=b(27)=0$. Hence, for a complete evaluation of $d(v)$, it is sufficient to determine the values $a(p^d)$ for every prime power $p^d$, and the values $b(3^d)$ for every $d\geq4$.

Our enumeration results are based on the following fact from \cite{KepNem}, Lemma 12.3.

\begin{prop}[Kepka \& N\v{e}mec] \label{thm:kepka-nemec}
Let $G_1$ and $G_2$ be commutative Moufang loops, $f$ a nuclear automorphism of $G_1$ and $g$ a nuclear automorphism of $G_2$ such that both $I-f$ and $I-g$ are automorphisms. Then ${\rm Aff}(G_1,f)\cong{\rm Aff}(G_2,g)$ if and only if there exists a group isomorphism $\psi:G_1\simeq G_2$ where $g=\psi f\psi^{-1}$.
\end{prop}

We start with the enumeration of affine Mendelsohn quasigroups of  prime order or prime squared order.

\begin{thm}\label{thm:enumeration} 
Let $p$ be a~prime.
\begin{enumerate}[(i)]
	\item If $p \equiv 1 \pmod 3$, then $a(p)=2$ and $a(p^2)=5$.
	\item If $p \equiv 2 \pmod 3$, then $a(p)=0$ and $a(p^2)=1$.
	\item $a(3)=1$ and $a(9)=2$.
\end{enumerate}
\end{thm}

\begin{proof}
First consider the prime orders. In this case any affine Mendelsohn quasigroup of order $p$ is isomorphic to $\Aff(\mathbb Z_p,k)$ where $k\in\mathbb Z_p^*$ is a root of the polynomial $f=x^2-x+1$ modulo $p$. Since $\mathbb Z_p^*$ is commutative, different roots $k$ result in non-isomorphic quasigroups by Proposition \ref{thm:kepka-nemec}. The number of roots was determined in Lemma \ref{lem:roots_of_f}.

For prime squared order, there are two possibilities. If $G=\mathbb Z_{p^2}$, we proceed similarly, reading the number of roots of $f$ modulo $p^2$ in Lemma \ref{lem:roots_of_f}. Let $G=(\mathbb Z_p)^2$. Its automorphism group is $GL(2,\mathbb F_p)$, hence we need to determine the number of conjugacy classes of matrices $A$ satisfying $f(A)=A^2 - A + I = 0$. For $p \equiv 1 \pmod 3$ and $p = 3$, $f$ splits over $\mathbb F_p$, hence such matrices are determined by their Jordan normal form. The key observation here is that if matrix $A$ satisfies $f(A)=0$, then every eigenvalue of $A$ is a root of $f$.

For $p = 3$, there are two possible Jordan forms
\[
  \begin{pmatrix}
    2 & 0 \\
     0 & 2 \\
  \end{pmatrix},\quad
  \begin{pmatrix}
    2 & 1 \\
     0 & 2 \\
  \end{pmatrix}.
\]
Both matrices satisfy the equality $f(A)=0$. Thus the number of isomorphism classes of affine Mendelsohn quasigroups with base group $\mathbb{Z}_{9}$ is 0 and with base group $(\mathbb{Z}_3)^2$ is 2. Hence $a(9)=0+2=2$.

For $p \equiv 1\,(\text{mod } 3)$, let $\tau_1$, $\tau_2$ be the two distinct roots of~$f$ over $\mathbb F_p$. There are five possibilities
\[
  \begin{pmatrix}
    \tau_1 & 0 \\
     0 & \tau_2 \\
  \end{pmatrix},\quad
  \begin{pmatrix}
    \tau_1 & 0 \\
     0 & \tau_1 \\
  \end{pmatrix},\quad
  \begin{pmatrix}
    \tau_2 & 0 \\
     0 & \tau_2 \\
  \end{pmatrix},\quad
  \begin{pmatrix}
    \tau_1 & 1 \\
     0 & \tau_1 \\
  \end{pmatrix},\quad
  \begin{pmatrix}
    \tau_2 & 1 \\
     0 & \tau_2 \\
  \end{pmatrix}.
\]
The former three matrices satisfy the equality $f(A)=0$, while the latter two matrices fail the equality. Thus when $p \equiv 1\,(\text{mod }3)$ the number of affine Mendelsohn quasigroups with base group $\mathbb{Z}_{p^2}$ is 2 and with base group $(\mathbb{Z}_p)^2$ is 3. Hence $a(p^2)=2+3=5$.

Finally consider the case $p \equiv 2 \pmod 3$. Let $A$ be a matrix satisfying $f(A)=0$. Since $f$ is irreducible over $\mathbb F_p$, $A$ has no eigenvector, hence $\{\mathbf v,A\mathbf v\}$ is a basis of $(\mathbb F_p)^2$, for any vector $\mathbf v\neq\mathbf 0$. Since $A(A\mathbf v) = A^2\mathbf v = (A - I)\mathbf v$, the matrix of the linear map given by $A$ in the basis $\{\mathbf v,A\mathbf v\}$ is
\[
  \begin{pmatrix}
    0 & -1 \\
    1 & 1 \\
  \end{pmatrix}.
\]
In particular, $A$ is conjugate to this matrix. Hence $a(p^2)=0+1=1$.
\end{proof}

Further values of $a(v)$ can be evaluated in GAP \cite{GAP} by a straightforward calculation using Proposition \ref{thm:kepka-nemec}. The values of $a(v)$ for prime powers $p^d<1000$ not covered by Lemma \ref{lem:necessary} and Theorem \ref{thm:enumeration} are summarised in Table \ref{tab:a_values}. 

\begin{table}[!hb]
\begin{center}
\begin{tabular}{c|ccccccccc}
$v$ & $2^4$ & $2^6$ & $2^8$ & $3^3$ & $3^4$ & $3^5$ & $3^6$ & $5^4$ & $7^3$\\
\hline
$a(v)$ & 2 & 3 & 5 & 3 & 5 & 7 & 11 & 2 & 10\\
\end{tabular}
\caption{Values of $a(v)$.}
\label{tab:a_values}
\end{center}
\end{table}

Commutative Moufang loops of orders 81 and 243 were classified by Kepka and N\v emec in \cite{KepNem}, Theorem 9.2, and the list is a part of the GAP package LOOPS \cite{loops} (we will use the notation of both \cite{KepNem} and LOOPS to refer to particular quasigroups and loops). 
Therefore we can (in theory) proceed similarly as in the affine case. A straightforward calculation shows that $b(81)=2$, with the loop $L(1)=\mathtt{MoufangLoop(81,1)}$ providing two distributive Mendelsohn quasigroups, $D(1)$ and $D(2)$, and the loop $L(2)=\mathtt{MoufangLoop(81,2)}$ none. For order 243, there is no distributive Mendelsohn quasigroup over the loops $L(i)$, $i=3,4,5,6$, that is $\mathtt{MoufangLoop(243,}i\mathtt{)}$ for $i=1,2,5,67$. For $L(2)\times\mathbb Z_3=\mathtt{MoufangLoop(243,57)}$ there is one distributive Mendelsohn quasigroup. For $L(1)\times\mathbb Z_3=\mathtt{MoufangLoop(243,56)}$, the automorphism group is too complicated and GAP fails to find the conjugacy classes; however, we see from the direct decomposition that the loop $\mathtt{MoufangLoop(243,56)}$ must provide at least two distributive quasigroups of order 243.
The numbers are summarised in Table \ref{tab:enumerationb}. An explicit description of the quasigroups of order 81 can be found in \cite{Sta}, Example 3.4.
The GAP code used for the calculations is available on our website\footnote{\tt{http://www.karlin.mff.cuni.cz/\~{}stanovsk/quandles}}.

\begin{table}[!hb]
\begin{center}
\begin{tabular}{c|ccccc}
$v$ & $3^1$ & $3^2$ & $3^3$ & $3^4$ & $3^5$ \\
\hline
$b(v)$ & 0 & 0 & 0 & 2 & $\geq3$\\
\end{tabular}
\caption{Values of $b(v)$.}
\label{tab:enumerationb}
\end{center}
\end{table}

Given an MTS($v$), $(V,\mathcal{B})$, its \textit{converse} is the MTS($v$), $(V,\mathcal{B}')$, obtained by writing all the blocks in the reverse order. In terms of the associated quasigroups, $Q=(V,\circ)$ and $Q'=(V,\circ')$, respectively, $Q'$ is the \emph{converse} of $Q$, i.e., $x\circ' y=y\circ x$.
A Mendelsohn triple system (respectively quasigroup) is not necessarily \textit{self-converse}, i.e., isomorphic to its converse.

\begin{prop}
Let $(V,\mathcal{B})$ be a Mendelsohn triple system such that the associated quasigroup $Q=\Aff(G,k)$ is distributive. Then $(V,\mathcal{B})$ is self-converse if and only if $k$ and $I-k$ are conjugate in $\mathrm{Aut}(G)$.
\end{prop}

\begin{proof}
The converse of $Q$ is the quasigroup $Q'=\Aff(G,I-k)$. According to Proposition \ref{thm:kepka-nemec}, $Q$ is isomorphic to $Q'$ if and only if there is an automorphism $\psi\in\mathrm{Aut}(G)$ such that $I-k=\psi k \psi^{-1}$, i.e., if and only if $k$ and $I-k$ are conjugate in $\mathrm{Aut}(G)$. 
\end{proof}

For example, if $G$ is a cyclic group, then $\mathrm{Aut}(G)$ is commutative, hence $Q$ is self-converse if and only if $k=I-k$, i.e., if and only if $Q$ is a Steiner quasigroup. In particular, a proper distributive MTS($p$) where $p$ is prime is never self-converse. Hence, as $a(p)=2$, the systems are the converse of each other.

\section{Anti-distributive Mendelsohn quasigroups}
\label{sec:antidistributive}

Both the\emph{ projective Steiner triple systems} and the \emph{Netto systems} are examples of anti-distributive (anti-mitre) Steiner triple systems:
\begin{itemize}
\item[(i)] Let $\mathbb{F}_2$ be the field of two elements and $V=(\mathbb{F}_2)^n\setminus\{\mathbf{0}\}$. Let $\mathcal{B}$ be the set of blocks $\{\mathbf{x},\mathbf{y},\mathbf{z}\}$ where $\mathbf{x},\mathbf{y},\mathbf{z}\in V$, $\mathbf{x}+\mathbf{y}+\mathbf{z}=\mathbf{0}$ and  $\mathbf{x}\neq\mathbf{y}\neq\mathbf{z}\neq\mathbf{x}$. This is the projective Steiner triple system PG($n-1,2$) of order $2^n-1$. The associated Steiner quasigroup is $((\mathbb{F}_2)^n\setminus\{\mathbf{0}\},\circ)$ where $\mathbf x\circ\mathbf y=\mathbf x+\mathbf y$ for $\mathbf x\neq\mathbf y$, and $\mathbf x\circ\mathbf x=\mathbf x$.
\item[(ii)] Let $p^d\equiv 7\,(\text{mod }12)$ where $p$ is prime. Let $\omega$ be a generator of the cyclic multiplicative group of order $p^d-1=12s+6$ of the Galois field $V=\mathbb F_{p^d}$. Let $\epsilon_1=\omega^{2s+1}$ and $\epsilon_2=\omega^{10s+5}$. Then $\epsilon_1\epsilon_2=\epsilon_1+\epsilon_2=1$. For $x,y\in V$ define $x<y$ if $y-x=\omega^i$ where $i$ is even. Either $x<y$ or $y<x$ but not both. Then the Netto system of order $p^d$ is determined by the Steiner quasigroup $(V,\circ)$ with $a\circ b=a\epsilon_1+b\epsilon_2$ whenever $a<b$, and $a\circ b=b\epsilon_1+a\epsilon_2$ whenever $b<a$. That is, the block containing the pair $\{a,b\}$ with $a<b$ is $\{a,b,a\epsilon_1+b\epsilon_2\}$. 
\end{itemize}
The above description of the Netto systems is due to Delandtsheer, Doyen, Siemons and Tamburini \cite{DelDoySieTam} and provides an interesting comparison to Lemma \ref{thm:menNet}. There, for $p^d=6s+1$ where $p$ is a prime, the affine MTS($p^d$) is constructed by defining the block containing the ordered pair $(a,b)$ to be $\langle a,b,a\epsilon_1+b\epsilon_2\rangle$ where $\epsilon_1=\omega^s$ and $\epsilon_2=\omega^{5s}$ where $\omega$ is a generator of the cyclic multiplicative group of order $p^d-1$. Here, in the Steiner triple system case, we must restrict our attention to when $p^d\equiv 7\,(\text{mod }12)$ so that an order can be assigned to the two elements $a$ and $b$ with $a\circ b$ being defined differently depending on whether $a<b$ or $b<a$. This `split case' is the essential reason behind why the Steiner case is anti-distributive while the Mendelsohn case is distributive.

The study of anti-mitre STS($v$) was begun in \cite{ColMenRosSir} and the existence spectrum, $v\equiv 1\text{ or }3\,(\text{mod }6)$, $v\neq 9$, was finally determined by Fujiwara \cite{Fuj1,Fuj2} and Wolfe \cite{Wol}.
\begin{thm}[Fujiwara \& Wolfe]
\label{thm:antimitre}
 An anti-mitre STS($v$) exists if and only if $v\equiv 1\text{ or }3 \,(\text{mod }6)$ and $v\neq 9$.
\end{thm}

Observe that, by taking an anti-mitre STS($v$) and writing each block in both of its two cyclic orders, we obtain a Mendelsohn triple system whose associated quasigroup is anti-distributive. Thus, as a consequence of Theorem \ref{thm:antimitre}, we obtain the following result.
\begin{thm}
There exists an MTS($v$) whose associated Mendelsohn quasigroup is  anti-distributive for all $v\equiv 1\text{ or }3\,(\text{mod }6)$ and $v\neq 9$.
\end{thm}
However our interest is in constructing proper Mendelsohn triple systems whose associated Mendelsohn quasigroups are anti-distributive. Such a result is obtained in Theorem 3.2, the proof of which may be simplified by considering the following lemma.

\begin{lem}
\label{lem:right_to_left}
Let $(V,\mathcal{B})$ be a Mendelsohn triple system with associated quasigroup $(V,\circ)$. Suppose that every ordered triple of distinct elements of $V$ that are not blocks in $\mathcal{B}$ violate right distributivity. Then they also violate left-distributivity, thus $(V,\mathcal{B})$ is anti-distributive.
\end{lem}

\begin{proof}
Consider an ordered triple of distinct elements $(x,y,z)$ where $x,y,z\in V$ and $\langle x,y,z\rangle\not\in \mathcal{B}$. Suppose that $(x,y,z)$ satisfies left distributivity, i.e.,
$$x\circ(y\circ z)=(x\circ y)\circ (x\circ z).$$
Then there exist $a,b,c,d\in V$ such that 
$\langle y,z,a\rangle, \langle x,a,b\rangle, \langle x,y,c\rangle, \langle x,z,d\rangle, \langle c,d,b\rangle\in\mathcal{B}.$
Note that, as $x$, $y$ and $z$ are all distinct, $x$, $c$ and $d$ are all distinct. Thus $$(c\circ d)\circ x=b\circ x=a=y\circ z=(c\circ x)\circ(d\circ x).$$
As $c$, $d$ and $x$ are all distinct, $\langle c,d,x\rangle \in \mathcal{B}$. So $c=z$ and $\langle x,y,z\rangle\in \mathcal{B}$, a contradiction.
\end{proof}

\begin{thm}
\label{thm:anti-dist}
There exists a proper MTS($v$) whose associated Mendelsohn quasigroup is anti-distributive for all $v\equiv 3\text{ or }7\,(\text{mod }12)$ except possibly for $v=19$.
\end{thm}

\begin{proof}
Let $(V,\mathcal{C})$ be an anti-mitre STS($v$) and let $(V,\star)$ be its associated Steiner quasigroup. For each block $\{a,b,c\}\in\mathcal{C}$ arbitrarily choose either the cyclic orientation $\langle a,b,c\rangle$ or the cyclic orientation $\langle a,c,b\rangle$. Once we have assigned these orientations, collectively the blocks have the property that every unordered pair, $\{x,y\}$, of distinct elements, occurs in a unique block either as the ordered pair $(x,y)$ or the ordered pair $(y,x)$. 
Without loss of generality, we assume that we chose the cyclic orientation $\langle a,b,c\rangle$ and we will denote the resulting collection of cyclically ordered blocks by $\mathcal{B}$. 

Let $V'=(V\times\{0,1\})\cup\{\infty\}$ and for ease of notation we will write $(a,j)\in (V\times\{0,1\})$ as $a_j$. Further we define the following set of cyclically ordered blocks 
$$\mathcal{B}'=\{\langle a_0,b_0,c_0\rangle, \langle a_1,b_1,c_0\rangle, \langle a_1,b_0,c_1\rangle, \langle a_0,b_1,c_1\rangle, \langle a_0,c_0,b_1\rangle, \langle a_0,c_1,b_0\rangle, \langle a_1,c_0,b_0\rangle,$$
$$\langle a_1,c_1,b_1\rangle:\langle a,b,c\rangle \in\mathcal{B}\} \,\cup\, \{\langle\infty,x_0,x_1\rangle, \langle\infty,x_1,x_0\rangle :x\in V\}.$$
We claim that the ordered pair $(V',\mathcal{B}')$ is a proper Mendelsohn triple system and that its associated Mendelsohn quasigroup is anti-distributive.

First we will show that $(V',\mathcal{B}')$ is a proper Mendelsohn triple system. Consider an unordered pair of elements from $V$, say $\{x,y\}$, then the ordered pairs $(x_0,y_0)$, $(x_0,y_1)$, $(x_1,y_0)$, $(x_1,y_1)$, $(y_0,x_0)$, $(y_0,x_1)$, $(y_1,x_0)$ and $(y_1,x_1)$ all occur in cyclically ordered blocks of $\mathcal{B}'$. Moreover, for all $x\in V$, the set $\mathcal{B}'$ contains cyclically ordered blocks which in turn contain the ordered pairs $(\infty,x_0)$, $(\infty,x_1)$, $(x_0,\infty)$, $(x_1,\infty)$, $(x_0,x_1)$ and $(x_1,x_0)$. Finally, as $(V,\mathcal{C})$ was a STS($v$) none of these ordered pairs appears more than once; hence, $(V',\mathcal{B}')$ is indeed a Mendelsohn triple system and it is easy to see that the system is proper.

Let $(V',\circ)$ be the associated Mendelsohn quasigroup of $(V',\mathcal{B}')$.
It remains to show that $(V',\circ)$ is anti-distributive. Thus, by Lemma \ref{lem:right_to_left}, showing that all ordered triples of distinct points in $V'$ which are not blocks of $\mathcal{B}'$ violate right distributivity completes the proof. We consider two cases, where $\infty$ is not an element of such an ordered triple and when $\infty$ is an element of such an ordered triple.
\begin{itemize}
\item[(1)] Suppose $(x_i,y_j,z_k)$ is an ordered triple of distinct elements, where $x,y,z\in V$, $i,j,k\in \{0,1\}$ and that $\langle x_i,y_j,z_k\rangle\not\in\mathcal{B}'$. Further suppose, for a contradiction, that $(x_i\circ y_j)\circ z_k=(x_i\circ z_k)\circ(y_j\circ z_k)$. Then $\{\langle x_i\circ y_j, z_k, (x_i\circ z_k)\circ(y_j\circ z_k)\rangle, \langle x_i,y_j,x_i\circ y_j\rangle, \langle x_i,z_k,x_i\circ z_k\rangle, \langle y_j,z_k,y_j\circ z_k\rangle, \langle x_i\circ z_k,y_j\circ z_k,(x_i\circ z_k)\circ(y_j\circ z_k)\rangle\}\subseteq\mathcal{B}'$, but this means that $\{x\star y, z, (x\star z)\star(y\star z)\}$, $\{x,y,x\star y\}$, $\{x,z,x\star z\}$, $\{y,z,y\star z\}$ and $\{x\star z,y\star z,(x\star z)\star(y\star z)\}$ are all blocks in $\mathcal{C}$, contradicting the fact that $(V,\mathcal{C})$ is an anti-mitre STS($v$).
\item[(2)] Suppose that $(a,b,c)$ is an ordered triple of distinct elements where $\langle a,b,c\rangle\not\in\mathcal{B}'$ and one of the following holds $(a,b,c)=(x_i,y_j,\infty)$ or  $(a,b,c)=(x_i,\infty,y_j)$ or $(a,b,c)=(\infty,x_i,y_j)$ where $x,y\in V$ and $i,j\in\{0,1\}$. We will consider these three cases separately. Note that in all three cases there exists   $z_k\in V\times\{0,1\}$ such that $\langle x_i,y_j,z_k\rangle\in\mathcal{B}'$. Subscript arithmetic is modulo 2.
\begin{itemize}
\item[(2.1)] Suppose that $(a,b,c)=(x_i,y_j,\infty)$. Then $(x_i\circ y_j)\circ \infty= z_k\circ \infty=z_{k+1}$ and $(x_i\circ \infty)\circ(y_j\circ \infty)=x_{i+1}\circ y_{j+1}=z_k\neq z_{k+1}$.
\item[(2.2)] Suppose that $(a,b,c)=(x_i,\infty,y_j)$.  Then $(x_i\circ \infty)\circ y_j= x_{i+1}\circ y_j=z_{k+1}$ and $(x_i\circ y_j)\circ(\infty\circ y_j)=z_k\circ y_{j+1}=x_{i}\neq z_{k+1}$.
\item[(2.3)] Finally suppose that $(a,b,c)=(\infty,x_i,y_j)$. Then $(\infty\circ x_i)\circ y_j= x_{i+1}\circ y_j=z_{k+1}$ and $(\infty\circ y_j)\circ(x_i\circ y_j)=y_{j+1}\circ z_k=x_{i+1}\neq z_{k+1}$.
\end{itemize}
\end{itemize}

From Theorem \ref{thm:antimitre} we know that an anti-mitre STS($v$) exists if and only if $v\equiv 1\text{ or }3 \,(\text{mod }6)$ and $v\neq 9$, and the result follows.
\end{proof}

The above theorem is a step towards establishing the existence spectrum for proper anti-distributive Mendelsohn triple systems. We expect that determining the entire spectrum, as was the case for anti-mitre Steiner triple systems, may be very difficult.

\end{document}